\documentclass{article}

\usepackage{amsmath}
\usepackage{amssymb}
\usepackage{graphicx}
\usepackage{bm}
\usepackage{mathdots}
\usepackage{booktabs}
\usepackage{rotating}
\usepackage[ruled,linesnumbered]{algorithm2e}
\usepackage[a4paper,left=2.8cm,right=2.8cm,top=2.5cm,bottom=2.5cm]{geometry}
\usepackage{fancyhdr}
\usepackage{hyperref}
\usepackage[framemethod=tikz]{mdframed}
\newtheorem{theorem}{Theorem}[section]
\newtheorem{corollary}{Corollary}[section]
\newtheorem{proposition}{Proposition}[section]
\newtheorem{lemma}{Lemma}[section]
\newtheorem{definition}{Definition}[section]

\newtheorem{remark}{Remark}[section]
\newtheorem{example}{Example}[section]

\makeatletter 
\@addtoreset{equation}{section}
\makeatother  

\newenvironment{proof}{{\noindent\it Proof.}\quad}{\hfill $\square$\\}
\usepackage{url}
\usepackage{mathtools}

\begin{document}
\title{The algebraic structure of hyperinterpolation class on the sphere}

\author{Congpei An\footnotemark[1]
       \quad \text{and}\quad Jiashu Ran\footnotemark[2]}
\renewcommand{\thefootnote}{\fnsymbol{footnote}}
\footnotetext[1]{School of Mathematics and Statistics, Guizhou University, Guiyang 550025, Guizhou, China (andbachcp@gmail.com)}
\footnotetext[2]{Department of Mathematical and Statistical Sciences, University of Alberta, Edmonton, Alberta T6G 2G1, Canada (jran@ualberta.ca),}
\maketitle

\begin{abstract}
This paper investigates the algebraic properties of the hyperinterpolation class $\mathbf{HC}(\mathbb{S}^d)$ on the unit sphere \( \mathbb{S}^d \). We focus on operators derived from the classical hyperinterpolation with bounded \( L_2 \) operator norms. By utilizing a discrete (semi) inner product framework, we develop the theory of hyper self-adjoint operators, hyper projection operators, and hyper semigroups. We analyze four specific operators: filtered, Lasso, hard thresholding, and generalized hyperinterpolations. We prove that the generalized hyperinterpolation operator is hyper self-adjoint and commutative with the hyperinterpolation operator. Additionally, we demonstrate that hard thresholding and classical hyperinterpolation operators form a hyper semigroup, with hard thresholding hyperinterpolation constituting the minimal prime hyper ideal. Finally, we establish that hyperinterpolation operators act as hyper homomorphisms on the hyper semigroup.
\end{abstract}



\section{Introduction}
In this paper, we explore the algebraic properties of hyperinterpolation class $\mathbf{HC}(\mathbb{S}^d)$ on the unit sphere $\mathbb{S}^d \subset \mathbb{R}^{d+1}$ and extend concepts and results from functional analysis into the realm of hyperinterpolation. Hyperinterpolation, initially introduced by Sloan in 1995 \cite{sloan1995hyperinterpolation}, offers a robust and efficient approximation for continuous functions in high-dimensional settings. It relies on high-order quadrature rules to approximate truncated Fourier expansions corresponding to orthonormal polynomial bases over compact subsets or manifolds \cite{sloan1995hyperinterpolation}.

Various forms of hyperinterpolation have been developed \cite{an2023hard,an2023hybrid,an2021lasso,dai2006GeneralizedHyperinterpolation,lin2021distributed,montufar2022distributed,pieper2009vector,reimer2002GeneralizedHyperinterpolation,sloan2012filtered}, broadening the applicability of hyperinterpolation across different scenarios. Recent advancements, such as relaxing or bypassing quadrature exactness assumptions via the Marcinkiewicz-Zygmund property \cite{an2022quadrature,an2024bypassing}, have further enhanced the utility of hyperinterpolation. Extensive research has been conducted on hyperinterpolation \cite{an2024efficient,caliari2007hyperinterpolation,caliari2008hyperinterpolation,de2009new,sommariva2014multivariate,atkinson2009norm,LeGia2001uniform,zbMATH01421286,sommariva2017nonstandard,sommariva2021sphericaltri,wang2014norm}. This paper focuses on the algebraic aspects of $\mathbf{HC}(\mathbb{S}^d)$ on the sphere $\mathbb{S}^d$, introducing concepts like hyper operator norms, hyper self-adjoint operators, and hyper projection operators.

Projection plays a major role in Hilbert spaces for certain problems, such as spectral theorem \cite{rudin2011functional}. Since hyperinterpolation operators are identified as projection operators that provide unique solutions to weighted least squares approximation problems \cite{sloan1995hyperinterpolation}, understanding their algebraic structure becomes crucial. We establish properties of hyper projection operators using a discrete (semi) inner product framework, including their behavior under product, sum, and difference operations. We also propose the concept of hyper semigroup over the set of continuous functions on the sphere, highlighting elements that form a semigroup and satisfy Pythagorean theorem. We prove that hard thresholding hyperinterpolation operators and hyperinterpolation operators belong to hyper semigroup.

Additionally, we delve into ideals and homomorphisms of $\mathbf{HC}(\mathbb{S}^d)$, identifying hyper semigroup formed by hard thresholding hyperinterpolation operators as the minimal prime hyper ideal of a broader hyper semigroup consisting of hard thresholding hyperinterpolation and hyperinterpolation operators. We  present the concept of hyper homomorphism and establish that hyperinterpolation operators serve as hyper homomorphisms on hyper semigroup.

In the sequel, we give necessary background information on polynomial spaces, quadrature rules and introduce some variants of hyperinterpolation.  After mathematical preliminaries in Section \ref{sec:setting}, a series of definitions pertinent to hyper projection is proposed in Section \ref{sec:projection}. As illustrated in Section \ref{sec:hyper_semigroup}, we present the concept of hyper semigroup and its properties. In Section \ref{sec4}, we analyze the operations on hyper projection operators, including the product, sum, and difference. Hyper ideals and hyper homomorphisms on hyper semigroups are defined and analyzed in Section \ref{sec:ideals_homomorphism}.

\section{Backgrounds on hyperinterpolation}\label{sec:setting}
Let $\mathbb{S}^{d}:=\left \{\mathbf{x} \in \mathbb{R}^{d+1}: \|\mathbf{x}\|_2 =1 \right \}$ with $d \geq 2$ be the unit sphere in $\mathbb{R}^{d+1}$, and let the space $L_2=L_2(\mathbb{S}^d)$ be the usual Hilbert space of square-integrable functions on
$\mathbb{S}^d$ with the inner product
\begin{equation*}
    \langle f,g \rangle_{L_2} = \int_{\mathbb{S}^d} f(\mathbf{x})g (\mathbf{x}){{{\rm{d}}\omega_d(\mathbf{x})}} \qquad \forall f,g \in L_2(\mathbb{S}^d),
\end{equation*}
and the induced norm $\|f\|_{L_2}:=\langle f,f \rangle^{1/2}_{L_2}$, where $\omega_d(\mathbf{x})$ denotes the Lebesgue surface measure on $\mathbb{S}^d$. The surface area of $\mathbb{S}^d$ is denoted by $\omega_d$,
\begin{equation*}
    \omega_d:=\frac{2\pi^{\frac{d+1}{2}}}{\Gamma\left(\frac{d+1}{2} \right)},
\end{equation*}
where $\Gamma(n):=(n-1)!$ for positive integers $n$. Let $\mathbb{P}_{n}(\mathbb{S}^d) \subset L_2(\mathbb{S}^d)$ be a linear space of all \emph{spherical polynomial} on $\mathbb{S}^d$ of total-degree not exceeding $n$. Note that the restriction of any real harmonic homogeneous polynomial on $\mathbb{R}^{d+1}$
of exact degree $\ell$ to $\mathbb{S}^d$ is called a real spherical harmonic of degree $\ell$. Let $\mathbb{H}_{\ell}(\mathbb{S}^d)$ be the space
of spherical harmonics of degree $\ell$ with the dimension $Z(d,\ell):=\dim(\mathbb{H}_{\ell}(\mathbb{S}^d))$  given by
\begin{equation*}
    Z(d,0):=1, \quad Z(d,\ell):=\frac{(2\ell + d-1)(\ell+d-2)!}{(d-1)!\ell!}, \quad \ell \in \mathbb{N}.
\end{equation*}
We denote by
\begin{equation}\label{equ:spherica_harmonics}
 \left \{ Y_{\ell k}^{(d)} : k=1,\cdots, Z(d,\ell) \right \}
\end{equation}
a given $L_2(\mathbb{S}^d)$-orthonormal system of real spherical harmonics of degree $\ell$ \cite{muller1966spherical}. Then \eqref{equ:spherica_harmonics}
is an $L_2(\mathbb{S}^d)$-orthonormal basis for $\mathbb{H}_{\ell}(\mathbb{S}^d)$, and the following holds
\begin{equation}\label{equ:direct_sum}
    \mathbb{P}_n(\mathbb{S}^d) =  \mathbb{H}_{0}(\mathbb{S}^d) \bigoplus \mathbb{H}_{1}(\mathbb{S}^d) \bigoplus \cdots \bigoplus \mathbb{H}_{n}(\mathbb{S}^d) 
\end{equation}
and
\begin{equation*}
    d_n:=\dim (\mathbb{P}_n(\mathbb{S}^d)) = \sum_{\ell=0}^{n} Z(d, \ell)= Z(d+1,n)=\frac{(2n+d)(n+d-1)!}{d!n!}.
\end{equation*}

The reproducing kernel $G_{\ell}: \mathbb{S}^{d}\times \mathbb{S}^{d}\to \mathbb{R}$ of the space $\mathbb{H}_{\ell}(\mathbb{S}^d)$ is
\begin{equation}\label{equ:kernel}
    G_{\ell}(\mathbf{x},\mathbf{y}):=\sum_{k=1}^{Z(d,\ell)} Y_{\ell k}^{(d)}(\mathbf{x})Y_{\ell k}^{(d)}(\mathbf{y}) = \frac{2\ell+d-1}{(d-1)\omega_d}C_{\ell}^{\frac{d-1}{2}}(\mathbf{x}\cdot\mathbf{y}),
\end{equation}
where $C_{\ell}^{\frac{d-1}{2}}$ is the Gegenbauer polynomial \cite{szego1975orthogonal} of degree $\ell$ and of index $\frac{d-1}{2}$, and $\mathbf{x}\cdot \mathbf{y}$ is the Euclidean inner product. Then $G_{\ell}$ satisfies the following properties:
\begin{equation*}
    G_{\ell}(\mathbf{x},\cdot) \in \mathbb{P}_{\ell}(\mathbb{S}^d), \quad \mathbf{x} \in \mathbb{S}^d; \qquad G_{\ell}(\mathbf{x},\mathbf{y})=G_{\ell}(\mathbf{y},\mathbf{x}), \quad \mathbf{x},\mathbf{y} \in \mathbb{S}^d;
\end{equation*}
\begin{equation*}
    p(\mathbf{x}) = \langle p, G_\ell(\mathbf{x},\cdot) \rangle_{L_2} , \quad \mathbf{x} \in \mathbb{S}^d, \quad p \in \mathbb{H}_{\ell}(\mathbb{S}^d).
\end{equation*}
Therefore, by \eqref{equ:direct_sum},
we have 
\begin{equation*}
    p = \langle p, E_n(\mathbf{x},\cdot) \rangle_{L_2} , \quad \forall\mathbf{x} \in \mathbb{S}^d, \quad \forall p \in \mathbb{P}_n(\mathbb{S}^d),
\end{equation*}
where 
\begin{equation}\label{equ:E_kernel}
    E_{n}(\mathbf{x},\mathbf{y}): = \sum_{\ell=0}^{n} G_{\ell}(\mathbf{x},\mathbf{y}) =\frac{1}{\omega_d} \frac{(d)_n}{\left (\frac{d}{2} \right )_n} P_n^{(\frac{d}{2},\frac{d-2}{2})}(\mathbf{x}\cdot\mathbf{y})
\end{equation}
with the Pochhammer symbol (rising factorial)
\begin{equation*}
    (a)_0:=1, \quad (a)_{\ell}:=a(a+1)\cdots (a+\ell -1), \quad a \in \mathbb{R}, \ell \in \mathbb{N},
\end{equation*}
and $P_n^{(\frac{d}{2},\frac{d-2}{2})}$ the Jacobi polynomial \cite{szego1975orthogonal} of degree $n$ with indices $\frac{d}{2}$ and $\frac{d-2}{2}$.

From $L_2(\mathbb{S}^d)$ to the polynomial space $\mathbb{P}_{n}(\mathbb{S}^d)$, there exists a unique projection $\mathcal{P}_{n}:L_2(\mathbb{S}^d) \to \mathbb{P}_{n}(\mathbb{S}^d)$, i.e.,
\begin{equation*}
    \mathcal{P}_{n} f:= \sum_{\ell=0}^{n} \sum_{k=1}^{Z(d,\ell)}\hat{f}_{\ell k}^{(d)} Y_{\ell k}^{(d)} = \sum_{\ell=0}^{n}\sum_{k=1}^{Z(d,\ell)} \left \langle f, Y_{\ell k}^{(d)} \right \rangle_{L_2} Y_{\ell k}^{(d)}, 
\end{equation*}
with the Fourier coefficients
\begin{equation}\label{equ:coefficients_fourier}
    \hat{f}_{\ell k}^{(d)} :=\left \langle f, Y_{\ell k}^{(d)} \right \rangle_{L^2} = \int_{\mathbb{S}^d} f(\mathbf{x}) Y_{\ell k}^{(d)}(\mathbf{x})  {{\rm{d}}\omega_d(\mathbf{x})}.
\end{equation}
With the help of $E_n$ of $\mathbb{P}_n(\mathbb{S}^d)$, we can rewrite the $L_2(\mathbb{S}^d)$-orthogonal
projection $\mathcal{P}_n$ onto $\mathbb{P}_n(\mathbb{S}^d)$ as
\begin{equation*}
    \mathcal{P}_n f (\mathbf{x}) =\int_{\mathbb{S}^d} f(\mathbf{y})E_n(\mathbf{x},\mathbf{y}) {\rm{d}}\omega_d(\mathbf{y}) =\left \langle f, E_n(\mathbf{x}, \cdot) \right \rangle_{L_2},  \quad \forall \mathbf{x} \in \mathbb{S}^d.
\end{equation*}

In the general case, the computation of the integral in \eqref{equ:coefficients_fourier} is a challenging task \cite{davis2007numerical}. Thus, to approximate the Fourier-like coefficients,  we seek the help of a quadrature rule
\begin{equation}\label{equ:quadrature}
    \sum_{j=1}^{N} w_j f(\mathbf{x}_j) \approx \int_{\mathbb{S}^d} f(\mathbf{x}) {\rm{d}}\omega_d(\mathbf{x}) \quad \forall f \in \mathcal{C}(\mathbb{S}^d),
\end{equation}
where $\mathbf{w}=\{w_1, \ldots, w_N\}$ is a set of positive weights and $\mathcal{X}_N=\{\mathbf{x}_1, \ldots, \mathbf{x}_N\} \subset \mathbb{S}^d$ is a set of nodes, and $\mathcal{C}(\mathbb{S}^{d})$ is the space of all continuous functions over $\mathbb{S}^d$. A quadrature rule \eqref{equ:quadrature} is said to have quadrature exactness $m$ if 
\begin{equation*}
     \sum_{j=1}^{N} w_j p(\mathbf{x}_j) = \int_{\mathbb{S}^d} p(\mathbf{x}) {\rm{d}}\omega_d(\mathbf{x}) \quad \forall p \in \mathbb{P}_m(\mathbb{S}^d).
\end{equation*}

Under a quadrature rule \eqref{equ:quadrature} with quadrature exactness $2n$, we introduce the ``\emph{discrete (semi) inner product}'' 
\begin{equation}\label{equ:semiinner}
    \langle  f,g \rangle_N := \sum_{j=1}^{N}w_jf(\mathbf{x}_j)g(\mathbf{x}_j) \quad\forall f,g \in \mathcal{C}(\mathbb{S}^d).
\end{equation} The \emph{hyperinterpolation operator} $\mathcal{L}_n: \mathcal{C}(\mathbb{S}^d) \to \mathbb{P}_n(\mathbb{S}^d)$ is defined \cite{sloan1995hyperinterpolation} as
\begin{eqnarray}\label{equ:hyper}
\mathcal{L}_n f:= \sum_{\ell=0}^{n}\sum_{k=1}^{Z(d,\ell)} \left \langle {f, Y_{\ell k}^{(d)}} \right \rangle_N   Y_{\ell k}^{(d)}. 
\end{eqnarray}

For more properties of hyperinterpolation on the sphere, we refer to \cite{an2024bypassing,MR2274179,dai2013approximation,MR2274179,reimer2003multivariate,sloan2012filtered,zbMATH01421286}. 
\subsection{Variants of hyperinterpolation}
To illustrate the diverse algebraic properties of hyperinterpolation, we review some variants of hyperinterpolation, such as \emph{filtered hyperinterpolation} \cite{sloan2012filtered},
\emph{Lasso hyperinterpolation} \cite{an2021lasso}, \emph{hard thresholding hyperinterpolation} \cite{an2023hard} and \emph{generalized hyperinterpolation} \cite{dai2006GeneralizedHyperinterpolation}. 
In contrast to those single-operator variants of hyperinterpolation, we have constructed a variant called \emph{hybrid hyperinterpolation} which is a composition
of filtered hyperinterpolation and Lasso hyperinterpolation \cite{an2023hybrid}. In this paper, we mainly investigate single operator variants of hyperinterpolation.  

Now, we first list some important definitions which are crucial in variants of hyperinterpolation. The following definitions of soft and hard thresholding operators are closely 
related to Lasso hyperinterpolation and hard thresholding hyperinterpolation, respectively. 

\begin{definition}[\cite{donoho1994ideal}]\label{def:softoperator}
The \emph{soft thresholding operator}, denoted by $\eta_S(a,k)$, is defined as
\begin{equation*}
 \eta_{S}(a,k):=\max(0,a-k)+\min(0,a+k)   .
\end{equation*}
\end{definition}

\begin{definition}[\cite{donoho1994ideal}]\label{def:hardoperator}
The \emph{hard thresholding operator}, denoted by $\eta_{H}(a,k)$, is defined as
\begin{equation*}
\eta_{H}(a,k):=\left\{\begin{array}{ll}
a,  & \text{if}~~|a|> k ,\\
0,  & \text{if}~~|a|\leq k.
\end{array}\right.
\end{equation*}
\end{definition}    

Then we can introduce \emph{Lasso hyperinterpolation} as the following.
\begin{definition}[\cite{an2021lasso}]\label{def:lassohyper}
Given a quadrature rule \eqref{equ:quadrature} with exactness $2n$, a \emph{Lasso hyperinterpolation} of $f \in \mathcal{C}(\mathbb{S}^d)$ onto $\mathbb{P}_{n}(\mathbb{S}^d)$ is defined as
\begin{equation}\label{equ:lassohyper}
    \mathcal{L}_n^{\lambda} f:= \sum_{\ell=0}^{n}\sum_{k=1}^{Z(d,\ell)} \eta_{S} \left( \left \langle f, Y_{\ell k}^{(d)} \right \rangle_N, \lambda\mu_{\ell k}  \right)Y_{\ell k}^{(d)},
\end{equation}
where $\lambda>0$ is the regularization parameter and $\mu_{\ell k}$'s are a set of positive penalty parameters.    
\end{definition}

Similarly, we can define \emph{hard thresholding hyperinterpolation} with the aid of
a hard thresholding operator $\eta_H(\cdot,\lambda)$.
\begin{definition}[\cite{an2023hard}]\label{def:hardhyper}
Given a quadrature rule \eqref{equ:quadrature} with exactness $2n$, {a \emph{hard thresholding hyperinterpolation}} of $f\in \mathcal{C}(\mathbb{S}^d)$ onto $\mathbb{P}_n(\mathbb{S}^d)$ is defined as
\begin{equation}\label{equ:hardhyper}
    \mathcal{H}_n^{\lambda} f:= \sum_{\ell=0}^{n}\sum_{k=1}^{Z(d,\ell)} \eta_{H} \left( \left \langle f, Y_{\ell k}^{(d)} \right \rangle_N, \lambda  \right)Y_{\ell k}^{(d)},
\end{equation}  
where $\lambda>0$ is the regularization parameter.
\end{definition}

\begin{figure}[htbp]
    \centering
    \includegraphics[scale=0.20]{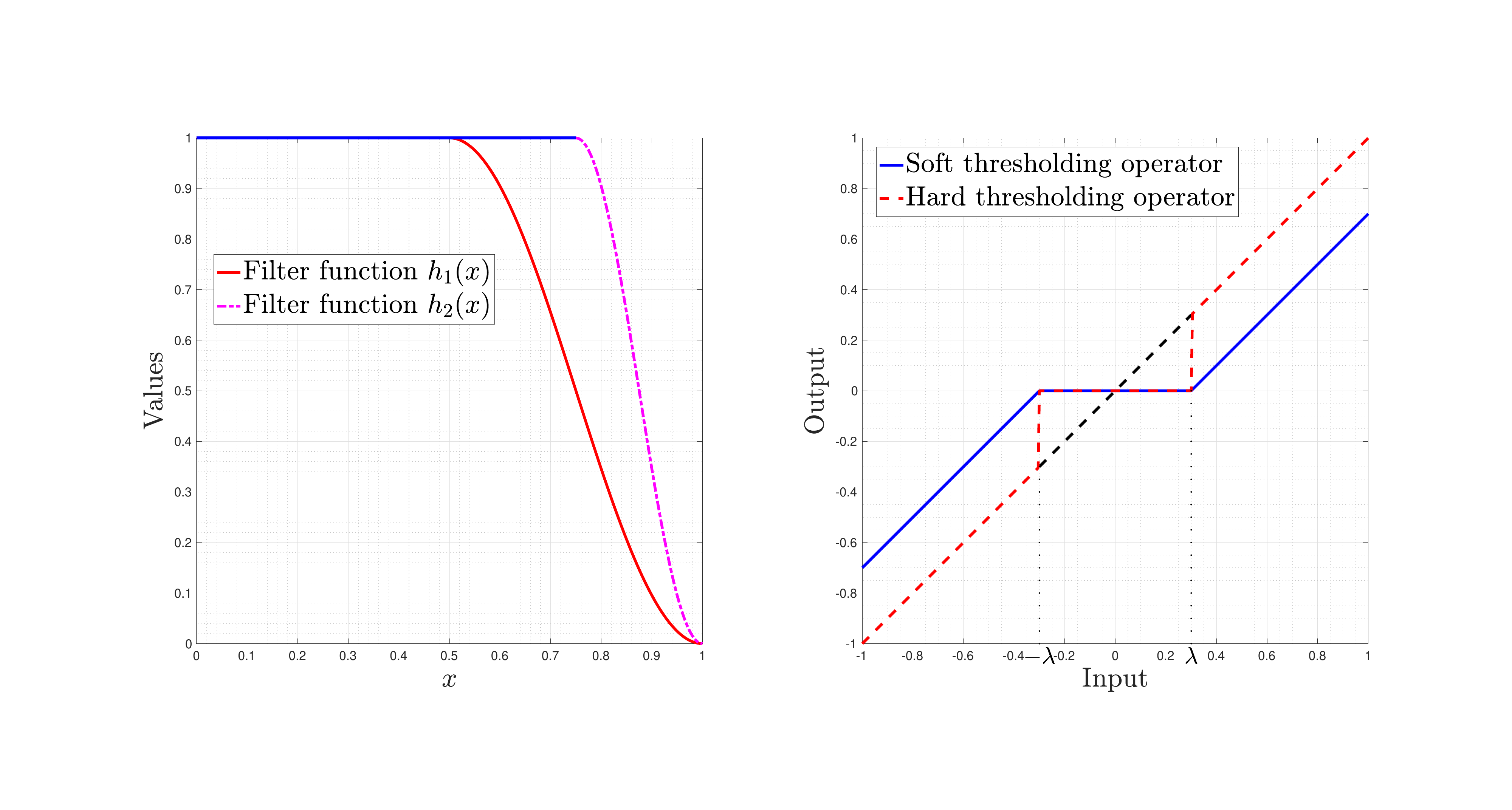}\\
    \caption{Left: filter functions; Right: soft and hard thresholding operators.}\label{fig:filter_soft_hard}
  \end{figure}

Next, we present the filter function utilized in filtered hyperinterpolation \cite{sloan2012filtered}. The \emph{filter function}, denoted by $h(x)$, is defined as
\begin{equation}\label{def:filterfunction}
h(x):=\left\{\begin{array}{ll}
1,  & \text{if}~~x \in [0,s] ,\\
0,  & \text{if}~~x\in[1,\infty),
\end{array}\right.
\end{equation}    
where $ 0<s<1 $ is a constant. If the sum of  $(d+1)$st forward difference of $h(x)$ is bounded, the corresponding approximation scheme is uniformly convergent \cite{sloan2011uniform}.

The filter function $h(x)$ has many different forms \cite{sloan2011polynomial}. In the following, we give two 
filter functions defined by sin function as examples \cite{an2011distribution}, i.e., when $s=\frac{1}{2}$,
\begin{equation}
    h_1(x):= \left\{\begin{array}{ll}
        1, & x\in[0,1/2], \\
        \sin^2 (\pi x), & x \in[1/2,1], \\
        0, & x \in [1,+\infty),
    \end{array}\right.
\end{equation}
and when $s=\frac{3}{4}$,
\begin{equation}
    h_2(x):= \left\{\begin{array}{ll}
        1, & x\in[0,3/4], \\
        \sin^2 (2\pi x), & x \in[3/4,1], \\
        0, & x \in [1,+\infty),
    \end{array}\right.
\end{equation}
and we display $h_1(x)$ and $h_2(x)$ on the left of Figure \ref{fig:filter_soft_hard}.

In what follows, we denote by $\lfloor \cdot \rfloor$ and $\lceil \cdot \rceil$ respectively the floor and the ceiling function. Given a chosen filter, we can define the filtered hyperinterpolation as the following:

\begin{definition}[\cite{sloan2012filtered}]\label{def:filteredhyper}
Let ${\overline{{\rm{supp}}(h)}}=[0,a]$, with $0<s< a \leq 1$. Given a quadrature rule \eqref{equ:quadrature} with exactness $\lceil an\rceil-1+\lfloor n/2 \rfloor$, the \emph{filtered hyperinterpolation} $\mathcal{F}_{n,N}{f} \in \mathbb{P}_{\lceil an\rceil-1} (\mathbb{S}^{d})$ of $f \in \mathcal{C}(\mathbb{S}^{d})$ is defined as
\begin{equation}\label{equ:filteredhyper}
    \mathcal{F}_{n,N}{f}:=\sum_{\ell=0}^{n}h\left(\frac{\ell}{n}\right)\sum_{k=1}^{Z(d,\ell)} \left \langle f, Y_{\ell k}^{(d)} \right \rangle_N Y_{\ell k}^{(d)}.
\end{equation}
where $h$ is a filter function defined by  \eqref{def:filterfunction}.
\end{definition}
Consequently, we have
\begin{proposition}   
For $0<s<1$,
\begin{equation*} 
    \mathcal{F}_{n,N}p = p  \quad \forall p \in \mathbb{P}_{\lfloor sn \rfloor}(\mathbb{S}^d).
\end{equation*}
\end{proposition}


In the following, we introduce the generalized hyperinterpolation:

\begin{definition}[\cite{dai2006GeneralizedHyperinterpolation}]\label{def:GeneralizedHyper}
    Let $G_{\ell}(\cdot)$ be defined by \eqref{equ:kernel} for $\ell=0,1,\ldots, n$. Assume that 
    \begin{equation*}
        D_n (t) = \sum_{\ell=0}^{n} a_{n\ell} G_{\ell}(t), \quad n=1,2,\ldots, \quad t \in [-1,1],
    \end{equation*}
    is a sequence of polynomials on $[-1,1]$ satisfying $\int_{\mathbb{S}^d} D_n(\mathbf{x}^{\rm{T}}\mathbf{y}){\rm{d}}\omega_d(\mathbf{y})=1$ for $\mathbf{x} \in \mathbb{S}^d$ and $n=1,2,\ldots$, and 
    \begin{equation*}
        \sup_n \int_0^{\pi} (1+n\theta)^2 |D_n(\cos(\theta))|\sin^{d-1}\theta {\rm{d}}\theta \leq K < \infty.
    \end{equation*}
    Then the \emph{generalized hyperinterpolation operators} $\mathfrak{G}\mathfrak{L}_{n}$ associated to $\{D_n(t)\}_{n=1}^{\infty}$ are defined as
    \begin{equation}\label{equ:GeneralizedHyper}
        \mathfrak{G}\mathfrak{L}_{n} f(\mathbf{x}) \equiv \mathfrak{G}\mathfrak{L}_{n,D_n} f(\mathbf{x})= \sum_{j=1}^{N}w_j f(\mathbf{x}_j)D_{n}(\mathbf{x}\cdot \mathbf{x}_j)
        = \sum_{\ell=0}^{n}a_{n\ell} \sum_{k=1}^{Z(d,\ell)} \left \langle f, Y_{\ell k}^{(d)} \right \rangle_N Y_{\ell k}^{(d)} (\mathbf{x}).
    \end{equation}
\end{definition}
    
\begin{remark}
Different from the hyperinterpolation operator $\mathcal{L}_n$ which demands the quadrature exactness of $2n$, the generalized hyperinterpolation operator $\mathfrak{G}\mathfrak{L}_n$ requires that
the positive quadrature formula is exact of degree $n+1$ in \cite{dai2006GeneralizedHyperinterpolation,reimer2002GeneralizedHyperinterpolation}. In fact, the quadrature exactness $n+1$ is the lowest requirement for generalized hyperinterpolation to guarantee the uniform convergence. 
The quadrature exactness $2n$ does not influence the uniform convergence for generalized hyperinterpolation \cite[Definition 6.9, Theorem 6.34]{reimer2003multivariate}.
\end{remark}

\begin{remark}
    The generalized hyperinterpolation defined in \cite{reimer2002GeneralizedHyperinterpolation,reimer2003multivariate} assumes the positivity of the kernels $D_n$. Under this assumption, Reimer proved that the generalized hyperinterpolation operator possesses the important 
    property that the approximation error can be uniformly controlled by the modulus of continuity of the first order \cite{reimer2002GeneralizedHyperinterpolation}. Dai proved the uniform approximation error for $\mathfrak{G}\mathfrak{L}_{n}$ 
    can be controlled by the modulus of continuity of the second order in \cite{dai2006GeneralizedHyperinterpolation}, where he bypassed the assumption of the positivity
 of the kernels $D_n$.
\end{remark}
To sum up, we can regard the variants of hyperinterpolation in above as 
an operator $\Phi$ act on the ``approximated" Fourier coefficients:
\begin{equation}\label{Tn}\mathcal{T}_n f : = \sum_{\ell=0}^{n}\sum_{k=1}^{Z(d,\ell)} \Phi\left(\left \langle {f, Y_{\ell k}^{(d)}} \right \rangle_N \right)   Y_{\ell k}^{(d)} .
\end{equation}
As a matter of fact, the $L_2$ operator norm of $\mathcal{T}_n$ should be bounded by a constant. This is consistent with the original hyperinterpolation
\cite{sloan1995hyperinterpolation}.
Now, we can define the important concept \emph{hyperinterpolation class}.
\begin{definition}\label{def:hyper_class}
Let $\mathcal{T}_n$ be defined by \eqref{Tn}.
Given a quadrature rule \eqref{equ:quadrature} with exactness $2n$, for $f \in \mathcal{C}(\mathbb{S}^d)$, the \emph{hyperinterpolation  class} on the sphere $\mathbb{S}^{d}$ is a set of operators defined by
\begin{equation}\label{equ:hyper_class}
    \mathbf{HC}(\mathbb{S}^d):=\left \{ \mathcal{T}_{n}|\,
    \text{ there exists}\,\,  c_d > 0 \,\, \text{such that}\,
    \Vert \mathcal{T}_{n}f \Vert_{L_2}\leq  c_d\Vert f\Vert_{\infty} \right \},
\end{equation}
where $c_d$ is independent of $n$.
\end{definition}
Before leaving this subsection we remark that there are many choices of the operator $\Phi:\mathbb{R} \to \mathbb{R}$, such as soft and hard thresholding operators \cite{donoho1994ideal}. Lasso, hard thresholding, filtered and generalized hyperinterpolations are elements of  $\mathbf{HC}(\mathbb{S}^d)$.  In fact,  the $L_2$ operator norm of Lasso hyperinterpolation  satisfies $\|\mathcal{L}_{n}^{\lambda} f\|_{L_2} \leq \tau \|f\|_{\infty}$ with $\tau:=\tau(f) \in (0,1)$ \cite[Theorem 4.4]{an2021lasso}, hard thresholding hyperinterpolation admits $\|\mathcal{H}_{n}^{\lambda} f\|_{L_2} \leq c_d \|f\|_{\infty}$ \cite[Theorem 4]{an2023hard}, filtered hyperinterpolation  achieves uniform boundedness $\|\mathcal{F}_{n,N}f\|_{\infty} \leq c_d \|f\|_{\infty}$ implying $L_2$-boundedness $\|\mathcal{F}_{n,N}f\|_{L_2} \leq c_d \|f\|_{\infty}$ \cite[Theorem 2]{sloan2012filtered}, generalized hyperinterpolation exhibits $\| \mathfrak{G}\mathfrak{L}_{n} \|_{\infty} \leq c_d \|f\|_{\infty}$, which ensures $\| \mathfrak{G}\mathfrak{L}_{n} \|_{L_2} \leq c_d \|f\|_{\infty}$ \cite[Lemma 3.3]{dai2006GeneralizedHyperinterpolation}.

From now on, any operator that we are interested in belongs to the hyperinterpolation class $\mathbf{HC}(\mathbb{S}^d)$.

\begin{remark}
     For a given operator $\mathcal{T}:\mathcal{C}(\mathbb{S}^d) \to L_2 (\mathbb{S}^d)$, the $L_2$ operator norm mentioned above is defined as 
    \begin{equation}
        \|\mathcal{T}\|_{\rm{op}}:=\sup_{f\neq 0}\frac{\|\mathcal{T} f\|_{L_2}}{\|f\|_{\infty}}.
    \end{equation}
    
\end{remark}

\section{Projection}\label{sec:projection}
It is well known that the concept of \emph{projection} plays a vital role in bounded operators on an inner product space \cite{rudin2011functional}. Similarly, the inner product plays a vital role in a projection property Hilbert space \cite{sloan1995hyperinterpolation}. Thus, we focus first on the projection properties under the setting of hyperinterpolation.

Based on the discrete (semi) inner product \eqref{equ:semiinner}, we introduce an important quantity as follows.
For a given $f \in \mathcal{C}(\mathbb{S}^d)$,  a real-valued function $\|f\|_{\ell^{2}(\mathbf{w})}$ defined as 
\begin{equation}\label{equ:seminorm}
    \|f\|_{\ell^{2}(\mathbf{w})} :=\langle  f,f \rangle_N^{1/2}= \left( \sum_{j=1}^{N} w_j |f(\mathbf{x}_j)|^2 \right)^{1/2},
\end{equation}
is a \emph{semi-norm} \cite[p. 24]{rudin2011functional}, where  $\mathbf{w}=\{w_1, \ldots, w_N\}$ is a set of positive weights and $\mathcal{X}_N=\{\mathbf{x}_1, \ldots, \mathbf{x}_N\} \subset \mathbb{S}^d$ is a set of nodes.
\begin{remark}\label{re:1}
Note that for a continuous function $f$ on $\mathbb{S}^d$, $\|f\|_{\ell^{2}(\mathbf{w})}=0$ does not mean that $f \equiv 0$ because there is a function $f$ that does not vanish in $\mathbb{S}^d$ but its values at all nodes $\mathbf{x}_1, \ldots, \mathbf{x}_N$ are zeros. It is not hard to verify the positive homogeneity and the subadditivity of $\|\cdot\|_{\ell^{2}(\mathbf{w})}$.
\end{remark}

Sloan revealed that $\mathcal{L}_nf$ is the best discrete least squares approximation (weighted by quadrature weights) of $f$ at the quadrature points in \cite{sloan1995hyperinterpolation}, which implies the important result as the following:

\begin{lemma}\label{lem:BestApproximation}
For a given $f \in \mathcal{C}(\mathbb{S}^d)$ with hyperinterpolation coefficients $a_{\ell k} = \left \langle f, Y_{\ell k}^{(d)} \right \rangle_{N}$ for $\ell=0,\ldots, n$ and $k=1,\ldots, Z(d,\ell)$.
Then 
\begin{equation}\label{equ:BestApproximation}
    \| f- \mathcal{L}_n f\|_{\ell^{2}(\mathbf{w})} \leq \left\| f - \sum_{\ell=0}^{n}\sum_{k=1}^{Z(d,\ell )} c_{\ell k} Y_{\ell k}^{(d)} \right\|_{\ell^2(\mathbf{w})}
\end{equation}
for any real numbers $c_{\ell k}$. Moreover, equality holds precisely when $c_{\ell k}=a_{\ell k}$ for all $\ell=0,\ldots, n$ and $k=1,\ldots, Z(d,\ell)$.
\end{lemma}

\begin{proof}
By direct computation, we have
\begin{equation*}
    f - \sum_{\ell=0}^{n}\sum_{k=1}^{Z(d,\ell )} c_{\ell k} Y_{\ell k}^{(d)}  = f- \mathcal{L}_n f + \sum_{\ell=0}^{n}\sum_{k=1}^{Z(d,\ell )} b_{\ell k} Y_{\ell k}^{(d)},
\end{equation*}
where $b_{\ell k}= a_{\ell k} - c_{\ell k}$. Then by Pythagorean theorem, we complete the proof.
\end{proof}

In a normed space, the unique decomposition of an element holds \cite{rudin2011functional}. However, the situation 
is different when we consider the semi-norm $\| \cdot \|_{\ell^2(\mathbf{w})}$.

\begin{theorem}\label{thm:orthogonal-completion}
Let the continuous functions space $\mathcal{C}(\mathbb{S}^d)$ be equipped with the semi-norm $\| \cdot \|_{\ell^2(\mathbf{w})}$. For any closed subspace $M \subset \mathcal{C}(\mathbb{S}^d)$, there holds
\begin{equation*}
    M \cap M^{\perp}=\{f: \|f\|_{\ell^2(\mathbf{w})} =0 ,f \in M, f \in M^{\perp} \},
\end{equation*}
where $M^{\perp}=\{ g: \langle f, g \rangle_N = 0, f\in M, g \in M^{\perp} \}$. 
\end{theorem}

\begin{proof}
Let $f \in M$ and $f \in M^{\perp}$. Since $\|\cdot\|_{\ell^2(\mathbf{w})}$ is a semi-norm, then 
$\|f\|_{\ell^2(\mathbf{w})}^2=\langle f , f\rangle_N=0$ does not mean $f \equiv 0$ by Remark \ref{re:1}.    
\end{proof}

\subsection{Hyper projection operators}

With the help of the discrete (semi) inner product \eqref{equ:semiinner}, we introduce the following two concepts.
\begin{definition}\label{def:HyperAdjoint}
    An operator $\mathcal{T}^{\ast}$ is a \emph{hyper adjoint} operator of another operator $\mathcal{T}$ if
    \begin{equation}\label{equ:HyperAdjoint}
     \langle \mathcal{T}f , g \rangle_N =  \langle  f, \mathcal{T}^{\ast}g \rangle_N  \quad \forall f,g \in \mathcal{C}(\mathbb{S}^d).
    \end{equation}
   
\end{definition}
   
\begin{definition}\label{def:Self-adjoint}
An operator $\mathcal{T}$ is a \emph{hyper self-adjoint} operator in the sense of the discrete (semi) inner product $\langle \cdot, \cdot \rangle_N$, i.e.,
\begin{equation}\label{equ:Self-adjoint}
 \langle \mathcal{T}f , g \rangle_N =  \langle  g, \mathcal{T}f \rangle_N \quad \forall f,g \in \mathcal{C}(\mathbb{S}^d). 
\end{equation}
\end{definition}
These two concepts can be compared with the adjointness in Hilbert spaces \cite{rudin2011functional}. 

\begin{proposition}\label{prop:filter_hyper_adjoint}
Filtered hyperinterpolation operator $\mathcal{F}_{n,N}$ and hyperinterpolation operator $\mathcal{L}_n$ are hyper self-adjoint.  
\end{proposition}
\begin{proof}
Given $f,g \in \mathcal{C}(\mathbb{S}^d)$, we obtain
\begin{equation*}
    \begin{aligned}
        \left \langle \mathcal{F}_{n,N} f ,g \right \rangle_N &= \left \langle  \sum_{\ell=0}^{n}h_{\ell}\sum_{k=1}^{Z(d,\ell)}  \left \langle f, Y_{\ell k}^{(d)} \right \rangle_N Y_{\ell k} ^{(d)} , g\right \rangle_N   
        = \sum_{\ell=0}^{n}h_{\ell}\sum_{k=1}^{Z(d,\ell)} \left \langle f,Y_{\ell k}^{(d)} \right \rangle_N   \left \langle Y_{\ell k}^{(d)} ,g \right \rangle_N  \\
        &=\left \langle f,\,\, \sum_{\ell=0}^{n} h_{\ell} \sum_{k=1}^{Z(d,\ell)} \left \langle Y_{\ell k}^{(d)} , g  \right \rangle_N Y_{\ell k}^{(d)}   \right \rangle_N 
        = \left \langle f,\, \,\mathcal{F}_{n,N} g  \right \rangle_N .
    \end{aligned}
\end{equation*}
If we take $h_{\ell}=1$ for $\ell=0,\ldots,n$ in the above equation, then we get 
\begin{equation*}
    \left \langle \mathcal{L}_{n} f ,g \right \rangle_N  = \left \langle  f , \mathcal{L}_{n}g \right \rangle_N. 
\end{equation*}
Thus, we have completed the proof.
\end{proof}

The following theorem states that the generalized hyperinterpolation operators $\mathfrak{G}\mathfrak{L}_{n}$ are hyper self-adjoint and commutative
with the hyperinterpolation operator $\mathcal{L}_n$.
\begin{theorem}\label{thm:GeneralizedHyperSelfAdjoint}
Let $f,g \in \mathcal{C}(\mathbb{S}^d)$. The generalized hyperinterpolation operators $\mathfrak{G}\mathfrak{L}_{n}$ are hyper self-adjoint, i.e.,
\begin{equation*}
    \langle \mathfrak{G}\mathfrak{L}_{n} f, g \rangle_N = \langle f, \mathfrak{G}\mathfrak{L}_{n} g \rangle_N
\end{equation*}
and commutative with $\mathcal{L}_n$, 
\begin{equation*}
    \mathfrak{G}\mathfrak{L}_{n} (\mathcal{L}_n f)  = \mathcal{L}_n (\mathfrak{G}\mathfrak{L}_{n} f) = \mathfrak{G}\mathfrak{L}_{n} f.
\end{equation*}
\end{theorem}

\begin{proof}
By direct computation, we have
\begin{equation*}
    \langle \mathfrak{G}\mathfrak{L}_{n} f, g \rangle_N = \left \langle \sum_{\ell=0}^{n}a_{n\ell} \sum_{k=1}^{Z(d,\ell)} \langle f, Y_{\ell k}^{(d)} \rangle_N Y_{\ell k}^{(d)} , g \right \rangle_N
    =  \sum_{\ell=0}^{n}a_{n\ell} \sum_{k=1}^{Z(d,\ell)} \langle f, Y_{\ell k}^{(d)} \rangle_N  \langle  Y_{\ell k}^{(d)} ,g \rangle_N
\end{equation*}   
and 
\begin{equation*}
    \langle f, \mathfrak{G}\mathfrak{L}_{n} g \rangle_N =  \left \langle f,  \sum_{\ell=0}^{n}a_{n\ell} \sum_{k=1}^{Z(d,\ell)} \langle g, Y_{\ell k}^{(d)} \rangle_N Y_{\ell k}^{(d)}  \right \rangle_N
    = \sum_{\ell=0}^{n}a_{n\ell} \sum_{k=1}^{Z(d,\ell)} \langle f, Y_{\ell k}^{(d)} \rangle_N  \langle g,  Y_{\ell k}^{(d)}  \rangle_N.
\end{equation*}

Next, we obtain
\begin{equation*}
    \begin{aligned}
        \mathfrak{G}\mathfrak{L}_{n} (\mathcal{L}_n f) &= \sum_{\ell=0}^{n}a_{n\ell} \sum_{k=1}^{Z(d,\ell)} \left \langle \sum_{\ell'=0}^{n}\sum_{k'=1}^{Z(d,\ell')} \langle f, Y_{\ell' k'}^{(d)} \rangle_N Y_{\ell' k'}^{(d)} , Y_{\ell k}^{(d)} \right \rangle_N Y_{\ell k}^{(d)} (\mathbf{x}), \\
        &= \sum_{\ell=0}^{n}a_{n\ell} \sum_{k=1}^{Z(d,\ell)} \sum_{\ell'=0}^{n}\sum_{k'=1}^{Z(d,\ell')} \langle f, Y_{\ell' k'}^{(d)} \rangle_N \left \langle  Y_{\ell' k'}^{(d)} , Y_{\ell k}^{(d)} \right \rangle_N Y_{\ell k}^{(d)} (\mathbf{x}) , \\
        &=\sum_{\ell=0}^{n}a_{n\ell} \sum_{k=1}^{Z(d,\ell)} \langle f, Y_{\ell k}^{(d)} \rangle_N  Y_{\ell k}^{(d)} (\mathbf{x}), 
    \end{aligned}
\end{equation*}
and
\begin{equation*}
    \begin{aligned}
        \mathcal{L}_n (\mathfrak{G}\mathfrak{L}_{n} f) & = \sum_{\ell=0}^{n}\sum_{k=1}^{Z(d,\ell)} \left \langle \sum_{\ell'=0}^{n} a_{n\ell'}\sum_{k'=1}^{Z(d,\ell')} \langle f, Y_{\ell' k'}^{(d)} \rangle_N Y_{\ell' k'}^{(d)}, Y_{\ell k}^{(d)} \right \rangle_N Y_{\ell k}^{(d)} (\mathbf{x}),\\
        &=\sum_{\ell=0}^{n}\sum_{k=1}^{Z(d,\ell)}  \sum_{\ell'=0}^{n} a_{n\ell'}\sum_{k'=1}^{Z(d,\ell')} \langle f, Y_{\ell' k'}^{(d)} \rangle_N \left \langle  Y_{\ell' k'}^{(d)}, Y_{\ell k}^{(d)} \right \rangle_N Y_{\ell k}^{(d)} (\mathbf{x}),\\
        &=\sum_{\ell=0}^{n}a_{n\ell} \sum_{k=1}^{Z(d,\ell)} \langle f, Y_{\ell k}^{(d)} \rangle_N  Y_{\ell k}^{(d)} (\mathbf{x}).
    \end{aligned}  
\end{equation*}
Thus, we have completed the proof.
\end{proof}

Then we can define the \emph{hyper projection} operator which parallels with 
the definition of projection \cite{rudin2011functional}. 

\begin{definition}\label{def:hyper-projection}
An operator $\mathcal{T}$ is a \emph{hyper projection} operator, if it satisfies
\begin{enumerate}
    \item[(i)] $\mathcal{T}$ is {hyper self-adjoint};
    \item[(ii)] $\mathcal{T}^2= \mathcal{T}$;
\end{enumerate}
\end{definition}

\begin{theorem}\label{prop:bounded-self-adjoint}
Hyperinterpolation operator $\mathcal{L}_n$ is a hyper projection operator.
\end{theorem}
\begin{proof}
From Proposition \ref{prop:filter_hyper_adjoint}, we know that $\mathcal{L}_n$ is hyper self-adjoint.
It is easy to check that $\mathcal{L}_{n}(\mathcal{L}_n)=\mathcal{L}_n$.  Thus, we complete the proof.
\end{proof}
    
\begin{remark}
Filtered hyperinterpolation operator $\mathcal{F}_{n,N}$ fails to be a hyper projection operator due to $\mathcal{F}_{n,N} \mathcal{F}_{n,N} \neq  \mathcal{F}_{n,N}$. Similarly,
hard thresholding hyperinterpolation operator $\mathcal{H}_{n}^{\lambda}$ is not a hyper projection operator as it lacks  
hyper self-adjointness. Lasso hyperinterpolation operator $\mathcal{L}_{n}^{\lambda}$ fails to be both hyper self-adjoint and idempotent. Consequently, it does not qualify as a hyper projection operator.
\end{remark}

\begin{remark}
    The generalized hyperinterpolation operators $\mathfrak{G}\mathfrak{L}_{n}$ are not idempotent. Indeed, we have
    \begin{equation*}
        \begin{aligned}
            \mathfrak{G}\mathfrak{L}_{n} (  \mathfrak{G}\mathfrak{L}_{n}f(\mathbf{x})) &= \sum_{\ell=0}^{n}a_{n\ell} \sum_{k=1}^{Z(d,\ell)} \left \langle \sum_{\ell'=0}^{n}a_{n\ell'} \sum_{k'=1}^{Z(d,\ell')} \langle f, Y_{\ell' k'}^{(d)} \rangle_N Y_{\ell' k'}^{(d)} (\mathbf{x}), Y_{\ell k}^{(d)} \right \rangle_N Y_{\ell k}^{(d)} (\mathbf{x}),\\
            &= \sum_{\ell=0}^{n}a_{n\ell} \sum_{k=1}^{Z(d,\ell)} \sum_{\ell'=0}^{n}a_{n\ell'} \sum_{k'=1}^{Z(d,\ell')} \langle f, Y_{\ell' k'}^{(d)} \rangle_N \left \langle  Y_{\ell' k'}^{(d)} (\mathbf{x}), Y_{\ell k}^{(d)} \right \rangle_N Y_{\ell k}^{(d)} (\mathbf{x}),\\
            &= \sum_{\ell=0}^{n}a_{n\ell}^2 \sum_{k=1}^{Z(d,\ell)} \langle f, Y_{\ell k}^{(d)} \rangle_N Y_{\ell k}^{(d)} (\mathbf{x}),
        \end{aligned}
    \end{equation*}
    where $a_{n\ell} =1 $ does not hold for all $n=0,1, \ldots,$ and $\ell=0, \ldots, n$. Thus,  $\mathfrak{G}\mathfrak{L}_{n}$ can not be a hyper projection operator.
\end{remark}


\subsection{Hyper semigroup}\label{sec:hyper_semigroup}

\noindent Next, we introduce the \emph{hyper semigroup} $\mathfrak{B}(\mathbb{S}^d)$ as follows.

\begin{definition}\label{hypersemigroup}
Let $ \mathfrak{B}(\mathbb{S}^d)$ be a set of all operators $\mathcal{T}: \mathcal{C}(\mathbb{S}^d) \to \mathbb{P}_n(\mathbb{S}^d)$. We call  $\mathfrak{B}(\mathbb{S}^d)$ a \emph{hyper semigroup} if:
\begin{enumerate}
    \item \textbf{(Semigroup Structure)} For all $\mathcal{T}_1, \mathcal{T}_2 \in \mathfrak{B}$, the composition $\mathcal{T}_1 \circ \mathcal{T}_2$ belongs to $\mathfrak{B}$ and composition is associative:
        \[
        (\mathcal{T}_1 \circ \mathcal{T}_2) \circ \mathcal{T}_3 = \mathcal{T}_1 \circ (\mathcal{T}_2 \circ \mathcal{T}_3).
        \]
    \item \textbf{(Pythagorean Identity)} For all $\mathcal{T} \in \mathfrak{B}$ and $f \in \mathcal{C}(\mathbb{S}^d)$:
        \[
        \langle \mathcal{T}f, \mathcal{T}f \rangle_N + \langle \mathcal{T}f - f, \mathcal{T}f - f \rangle_N = \langle f, f \rangle_N.
        \]
\end{enumerate}
\end{definition}

\begin{remark}\label{rem:semigroup}
The concept of a \emph{hyper semigroup} is well-defined and distinct from classical algebraic structures:
\begin{itemize}
    \item \textbf{Semigroup vs. Algebra}: Unlike a complex algebra \cite[Definition 10.1]{rudin2011functional}, a hyper semigroup requires \emph{composition closure} and \emph{associativity} but not a vector space structure (i.e., no addition/scalar multiplication axioms). 
    \item \textbf{Non-emptiness}: $\mathfrak{B}$ contains key operators like the hard thresholding hyperinterpolation $\mathcal{H}_n^{\lambda}$ and hyperinterpolation $\mathcal{L}_n$ \cite[Lemma 3.1]{an2023hard}, \cite[Lemma 5]{sloan1995hyperinterpolation}, which satisfy both the Pythagorean identity and semigroup axioms under composition.
    \item \textbf{Composition Closure}: For $\mathcal{T}_a, \mathcal{T}_b \in \{\mathcal{L}_n, \mathcal{H}_n^\lambda\}$, composites reduce to $\mathcal{H}_{\min(m,n)}^\lambda$ or $\mathcal{L}_{\min(m,n)}$ (see Theorem \ref{thm:ideal} and equations \eqref{equ:prime_ideal_1}--\eqref{equ:prime_ideal_2}).
\end{itemize}
\end{remark}

\begin{proposition}\label{prop:1}
If $\mathcal{T} \in \mathfrak{B}(\mathbb{S}^d)$, then the following holds
\begin{enumerate}
    \item[(i)] $\langle \mathcal{T}f,f \rangle_N = \| \mathcal{T} f \|_{\ell^2(\mathbf{w})}^2$;
    \item[(ii)] $\| \mathcal{T} f \|_{\ell^2(\mathbf{w})} \leq \| f\|_{\ell^2(\mathbf{w})}$.
\end{enumerate}
\end{proposition}

\begin{proof}
(i) Since $\mathcal{T}$ satisfies Pythagorean theorem, it is clear that 
\begin{equation*}
    2\langle \mathcal{T}f,\mathcal{T}f \rangle_N - 2\langle \mathcal{T}f,f \rangle_N + \langle f, f \rangle_N = \langle f, f \rangle_N,
\end{equation*}
which implies that
\begin{equation*}
    \langle \mathcal{T}f,f \rangle_N = \| \mathcal{T} f \|_{\ell^2(\mathbf{w})}^2.
\end{equation*}
(ii) From (i), we have
\begin{equation*}
    \| \mathcal{T} f \|_{\ell^2(\mathbf{w})}^2 = \langle \mathcal{T}f,f \rangle_N \leq \| f\|_{\ell^2(\mathbf{w})} \|\mathcal{T} f\|_{\ell^2(\mathbf{w})},
\end{equation*}
where the inequality follows from the Cauchy-Schwarz inequality. 
\end{proof}

\begin{lemma}\label{lem:norm1}
    Let $\mathcal{T} \in \mathfrak{B}(\mathbb{S}^d)$ and let $\mathcal{T}^2 = \mathcal{T}$. Then $\|\mathcal{T}\|_{\ell^2(\mathbf{w})}=1$,
   where \begin{equation}
       \|\mathcal{T}\|_{\ell^2(\mathbf{w})} := \sup_{\|f\|_{\ell^2(\mathbf{w})}\neq 0} \frac{\|\mathcal{T} f\|_{\ell^2(\mathbf{w})}}{\| f\|_{\ell^2(\mathbf{w})}}  = \sup_{\| f\|_{\ell^2(\mathbf{w})} =1 } \|\mathcal{T} f\|_{\ell^2(\mathbf{w})}.
\end{equation}
\end{lemma}
\begin{proof}
    By direct computation, we have
    \begin{equation*}
        \|\mathcal{T}f\|_{\ell^2(\mathbf{w})} = \|\mathcal{T}^2f\|_{\ell^2(\mathbf{w})} = \left\| \mathcal{T} \left( \frac{\mathcal{T}f}{ \left\| \mathcal{T}f \right\|_{\ell^{2}(\mathbf{w})}} \right) \right\|_{\ell^{2}(\mathbf{w})}  \left\|  \mathcal{T}f \right\|_{\ell^{2}(\mathbf{w})} \leq \|\mathcal{T}\|_{\ell^2(\mathbf{w})} \|\mathcal{T}f\|_{\ell^2(\mathbf{w})}.
    \end{equation*}
Then combining the result (ii) in Proposition \ref{prop:1}, we deduce that $\|\mathcal{T}\|_{\ell^2(\mathbf{w})}=1$.
\end{proof}

\begin{definition}\label{def:hyper_positive}
 An operator $\mathcal{T}$ is called a \emph{hyper nonnegative operator} if 
 \begin{equation*}
    \langle \mathcal{T}f, f \rangle_N \geq 0 \quad \forall f \in \mathcal{C}(\mathbb{S}^d).
 \end{equation*}  
Given $f,g \in \mathcal{C}(\mathbb{S}^d)$, $f$ and $g$ are called \emph{hyper orthogonal} if 
\begin{equation}\label{equ:hyper-orthogonal}
    \langle f, g \rangle_N = 0 .
\end{equation}   
\end{definition}

\begin{example}
Hyperinterpolation operator $\mathcal{L}_n$, hard thresholding hyperinterpolation operator $\mathcal{H}_n^{\lambda}$,
Lasso hyperinterpolation operator $\mathcal{L}_n^{\lambda}$ and filtered hyperinterpolation operator $\mathcal{F}_{n,N}$ are hyper nonnegative operators. However, 
the generalized hyperinterpolation operator $\mathfrak{G}\mathfrak{L}_n$ may not be a nonnegative operator. Indeed, for any $f \in \mathcal{C}(\mathbb{S}^d)$, we have
\begin{equation*}
    \langle \mathfrak{G}\mathfrak{L}_{n} f, f \rangle_N = \left \langle \sum_{\ell=0}^{n}a_{n\ell} \sum_{k=1}^{Z(d,\ell)} \langle f, Y_{\ell k}^{(d)} \rangle_N Y_{\ell k}^{(d)} , f \right \rangle_N
    =  \sum_{\ell=0}^{n}a_{n\ell} \sum_{k=1}^{Z(d,\ell)} \langle f, Y_{\ell k}^{(d)} \rangle_N^2.
\end{equation*}   
We cannot judge the sign of $\langle \mathfrak{G}\mathfrak{L}_{n} f, f \rangle_N$ due to the existence of $a_{n\ell}$'s.
\end{example}

\begin{theorem}\label{thm:11}
    If $\mathcal{T} \in \mathfrak{B}(\mathbb{S}^d)$ and  $\langle \mathcal{T} f, f \rangle_N =0 $ 
    for every $f \in \mathcal{C}(\mathbb{S}^d)$, then $\mathcal{T}=0$.
\end{theorem}

\begin{proof}
From Proposition \ref{prop:1} (i), we have
\begin{equation*}
    \langle \mathcal{T}f,f \rangle_N = \| \mathcal{T} f \|_{\ell^2(\mathbf{w})}^2=0, \quad \forall f \in \mathcal{C}(\mathbb{S}^d)
\end{equation*}
Thus, we deduce that $\mathcal{T}f=0$.
\end{proof}

\begin{corollary}
    If $\mathcal{T} \in \mathfrak{B}(\mathbb{S}^d)$, $\mathcal{S} \in \mathfrak{B}(\mathbb{S}^d)$,and
    \begin{equation*}
        \langle \mathcal{T}f,f \rangle_N = \langle \mathcal{S}f,f \rangle_N 
    \end{equation*}
    for every $f \in \mathcal{C}(\mathbb{S}^d)$, then $\mathcal{S}=\mathcal{T}$.
\end{corollary}

\section{Operations on hyper projection operators}\label{sec4}
\subsection{The product of hyper projection operators}
In a Hilbert space, we often decompose normal operators into projection operators, which is very important in the spectral theorem \cite{rudin2011functional}. Thus, we will inevitably encounter the operations on projection operators.
Similarly, we first discuss products of hyper projection operators:
\begin{theorem}\label{thm:product-projection}
For two hyper projection operators $\mathcal{T}_1$ and $\mathcal{T}_2$, the product $\mathcal{T}_1\mathcal{T}_2$ is also a hyper projection operators if and only if $\mathcal{T}_1$ and $\mathcal{T}_2$ are commutative.
\end{theorem}

\begin{proof}
Sufficiency. Given  $f,g \in \mathcal{C}(\mathbb{S}^d)$. If $\mathcal{T}_1$ and $\mathcal{T}_2$ are commutative, then 
\begin{equation*}
    \langle \mathcal{T}_1\mathcal{T}_2 f, g \rangle_N = \langle \mathcal{T}_2 f,  \mathcal{T}_1g \rangle_N =\langle  f,  \mathcal{T}_2\mathcal{T}_1g \rangle_N  = \langle  f,  \mathcal{T}_1\mathcal{T}_2g \rangle_N, 
\end{equation*}
which implies that $\mathcal{T}_1\mathcal{T}_2$ is hyper self-adjoint. It is clear that
\begin{equation*}
    (\mathcal{T}_1\mathcal{T}_2)^2=\mathcal{T}_1\mathcal{T}_2\mathcal{T}_1\mathcal{T}_2= (\mathcal{T}_1)^2 (\mathcal{T}_2)^2= \mathcal{T}_1\mathcal{T}_2.
\end{equation*}
Necessity. Since $\mathcal{T}_1\mathcal{T}_2$ is a hyper projection operator, then $\mathcal{T}_1\mathcal{T}_2 = (\mathcal{T}_1\mathcal{T}_2)^{\ast}$ and
\begin{equation*}
    \langle \mathcal{T}_1\mathcal{T}_2 f,g \rangle_N = \langle \mathcal{T}_2 f, \mathcal{T}_1g \rangle_N=\langle  f, \mathcal{T}_2\mathcal{T}_1g \rangle_N=\langle  f, \mathcal{T}_1\mathcal{T}_2g \rangle_N .
\end{equation*}
Therefore, $\mathcal{T}_1$ and $\mathcal{T}_2$ are commutative.
\end{proof}

\begin{example}
 For two hyperinterpolation operators $\mathcal{L}_n$ and $\mathcal{L}_m$ with $n \leq m$, 
 we have $\mathcal{L}_n\mathcal{L}_m=\mathcal{L}_m\mathcal{L}_n = \mathcal{L}_n$.   
\end{example}

\subsection{The sum of hyper projection operators}
\begin{theorem}\label{thm:sum}
For two hyper projection operators $\mathcal{T}_1$ and $\mathcal{T}_2$, the sum $\mathcal{T}_1 + \mathcal{T}_2$ is also a hyper projection operator if and only if
$\mathcal{T}_1 \mathcal{T}_2 =0$ (or $\mathcal{T}_2\mathcal{T}_1=0)$
\end{theorem}
\begin{proof}
Necessity. If $\mathcal{T}_1 + \mathcal{T}_2$ is a hyper projection operator, then
\begin{equation*}
    \begin{aligned}
        \mathcal{T}_1 + \mathcal{T}_2 &= (\mathcal{T}_1 + \mathcal{T}_2)^2 = (\mathcal{T}_1 )^2 + \mathcal{T}_1 \mathcal{T}_2 + \mathcal{T}_2 \mathcal{T}_1 + ( \mathcal{T}_2)^2 , \\
        &=\mathcal{T}_1 + \mathcal{T}_1 \mathcal{T}_2 + \mathcal{T}_2 \mathcal{T}_1  +\mathcal{T}_2,
    \end{aligned}
\end{equation*}
which implies that $\mathcal{T}_1 \mathcal{T}_2 + \mathcal{T}_2 \mathcal{T}_1  =0$. By direct computation, we have
\begin{equation*}
\left\{
    \begin{array}{c}
         \mathcal{T}_1\mathcal{T}_1 \mathcal{T}_2 + \mathcal{T}_1\mathcal{T}_2 \mathcal{T}_1  =0,  \\
         \mathcal{T}_1 \mathcal{T}_2\mathcal{T}_1 + \mathcal{T}_2 \mathcal{T}_1\mathcal{T}_1  =0, 
    \end{array}
\right.
\end{equation*}
which means that
\begin{equation*}
    \mathcal{T}_1 \mathcal{T}_2 = - \mathcal{T}_1\mathcal{T}_2 \mathcal{T}_1 = \mathcal{T}_2 \mathcal{T}_1.
\end{equation*}
Hence, $\mathcal{T}_1\mathcal{T}_2 = \mathcal{T}_2 \mathcal{T}_1=0$. 

Sufficiency. If $\mathcal{T}_1\mathcal{T}_2 = \mathcal{T}_2 \mathcal{T}_1=0$, then 
\begin{equation*}
    \mathcal{T}_1 + \mathcal{T}_2 = (\mathcal{T}_1 + \mathcal{T}_2)^2 = (\mathcal{T}_1 )^2 + \mathcal{T}_1 \mathcal{T}_2 + \mathcal{T}_2 \mathcal{T}_1 + ( \mathcal{T}_2)^2 
    = \mathcal{T}_1 + \mathcal{T}_2.
\end{equation*}
Thus, we have completed the proof.
\end{proof}

\begin{example}
Let $f \in \mathcal{C}(\mathbb{S}^d)$. Considering the following two hyper projection operators    
\begin{equation*}
    \mathcal{L}_{n}f = \sum_{\ell=0}^{n} \sum_{k=1}^{Z(d,\ell)} \left \langle f, Y^{(d)}_{\ell k} \right \rangle_N Y^{(d)}_{\ell k} \quad \text{and} \quad \mathcal{L}_{m-n}^{'}f = \sum_{\ell=n+1}^{m} \sum_{k=1}^{Z(d,\ell)} \left \langle f, Y^{(d)}_{\ell k} \right \rangle_N Y^{(d)}_{\ell k}, \quad (n<m),
\end{equation*}
then we have
\begin{equation*}
    \mathcal{L}_{n} + \mathcal{L}_{m-n}^{'}= \sum_{\ell=0}^{m} \sum_{k=1}^{Z(d,\ell)} \left \langle f, Y^{(d)}_{\ell k} \right \rangle_N Y^{(d)}_{\ell k}.
\end{equation*}
\end{example}

\subsection{The difference of hyper projection operators}
\begin{definition}\label{def:subprojection}
For two hyper projection operators $\mathcal{T}_1$ and $\mathcal{T}_2$, we call $\mathcal{T}_1$ is a \emph{suboperator} of $\mathcal{T}_2$, if the image $M_1$ of $\mathcal{T}_1$ is contained in the image $M_2$ of $\mathcal{T}_2$.
\end{definition}

\begin{lemma}\label{lem:suboperator}
For two hyper projection operators $\mathcal{T}_1$ and $\mathcal{T}_2$, $\mathcal{T}_1$ is a suboperator of $\mathcal{T}_2$ if and only if one of the two conditions is satisfied:
\begin{enumerate}
    \item[(i)] $\mathcal{T}_1 \mathcal{T}_2 = \mathcal{T}_2 \mathcal{T}_1  =\mathcal{T}_1$;
    \item[(ii)] Given $f \in \mathcal{C}(\mathbb{S}^d)$, $\|\mathcal{T}_1 f\|_{\ell^2(\mathbf{w})} \leq \|\mathcal{T}_2 f\|_{\ell^2(\mathbf{w})}$ .
\end{enumerate}
\end{lemma}

\begin{proof}
Necessity. For $f \in \mathcal{C}(\mathbb{S}^d)$, since $\mathcal{T}_1 f \in M_1$, where $M_1$ is the image of $\mathcal{T}_1$, we have $\mathcal{T}_2 \mathcal{T}_1 f = \mathcal{T}_1 f$. Thus, we obtain
\begin{equation*}
    \mathcal{T}_2 \mathcal{T}_1  = \mathcal{T}_1 .
\end{equation*}
Furthermore, we obtain
\begin{equation*}
    \langle  \mathcal{T}_2 \mathcal{T}_1 f, g \rangle_N = \langle  f, \mathcal{T}_1 \mathcal{T}_2g \rangle_N =\langle  \mathcal{T}_1 f, g \rangle_N = \langle f,  \mathcal{T}_1g \rangle_N ,
\end{equation*}
which implies that $\mathcal{T}_2 \mathcal{T}_1 =\mathcal{T}_1 \mathcal{T}_2 =  \mathcal{T}_1$.  By direct computation and Lemma \ref{lem:norm1}, we have
\begin{equation*}
    \| \mathcal{T}_1 f\|_{\ell^2(\mathbf{w})} =   \| \mathcal{T}_1 \mathcal{T}_2 f\|_{\ell^2(\mathbf{w})} \leq \| \mathcal{T}_1 \|_{\ell^2(\mathbf{w})}   \| \mathcal{T}_2 f\|_{\ell^2(\mathbf{w})}  = \| \mathcal{T}_2 f\|_{\ell^2(\mathbf{w})}.
\end{equation*}

Sufficiency. If $\mathcal{T}_1$ is not a suboperator of $\mathcal{T}_2$, then there exist an element $\psi_0 \in M_1$ but $\psi_0 \notin M_2$, where $M_1,M_2$ are the images of $\mathcal{T}_1$ and $\mathcal{T}_2$, respectively. Let $\hat{\psi}_0$ be the hyper orthogonal projection of $\psi_0$ in $M_2$. 
Then $\| \hat{\psi}_0\|_{\ell^2(\mathbf{w})} < \| {\psi}_0\|_{\ell^2(\mathbf{w})}$ and $\mathcal{T}_2 \psi_0 = \hat{\psi}_0$. With the aid of Proposition \ref{prop:1} (ii), we have 
\begin{equation*}
    \|\mathcal{T}_2 \psi_0 \|_{\ell^2(\mathbf{w})} = \|\hat{\psi}_0 \|_{\ell^2(\mathbf{w})} < \|{\psi}_0 \|_{\ell^2(\mathbf{w})} = \|\mathcal{T}_1 \psi_0 \|_{\ell^2(\mathbf{w})},
\end{equation*}
which contradicts the assumption. Therefore, $\mathcal{T}_1$ is a suboperator of $\mathcal{T}_2$.
\end{proof}

\begin{remark}
Since filtered hyperinterpolation operator $\mathcal{F}_{n,N}$, hard thresholding hyperinterpolation operator $\mathcal{H}_{n}^{\lambda}$ and Lasso hyperinterpolation operator $\mathcal{L}_n^{\lambda}$
are not hyper projection operators, they fail to be suboperators of hyperinterpolation operators even if satisfing the conditions (i) or (ii) in Lemma \ref{lem:suboperator}. 
\end{remark}

Next, we focus on the difference between two hyper projection operators:
\begin{theorem}\label{thm:difference}
For two hyper projection operators $\mathcal{T}_1$ and $\mathcal{T}_2$, the difference $\mathcal{T}_2 - \mathcal{T}_1$ is a hyper projection operator if and only if $\mathcal{T}_1$ is a suboperator of $\mathcal{T}_2$.
\end{theorem}

\begin{proof}
Necessity. If $\mathcal{T}_2 - \mathcal{T}_1$ is a hyper projection operator, then we have
\begin{equation*}
\langle \mathcal{T}_2 (\mathcal{T}_2 - \mathcal{T}_1) f, g \rangle_N = \langle (\mathcal{T}_2  - \mathcal{T}_2\mathcal{T}_1) f, g \rangle_N = \langle  f, (\mathcal{T}_2 - \mathcal{T}_1)\mathcal{T}_2g \rangle_N = \langle  f, (\mathcal{T}_2 - \mathcal{T}_1\mathcal{T}_2)g \rangle_N.
\end{equation*}
Thus we can deduce that $\mathcal{T}_1\mathcal{T}_2 = \mathcal{T}_2\mathcal{T}_1$. Besides, the following holds:
\begin{equation}\label{equ:difference1}
(\mathcal{T}_2 - \mathcal{T}_1)^2 = \mathcal{T}_2^2 - \mathcal{T}_2\mathcal{T}_1 - \mathcal{T}_1 \mathcal{T}_2 + \mathcal{T}_1^2 = \mathcal{T}_2 - \mathcal{T}_1,
\end{equation}
which implies that $\mathcal{T}_2\mathcal{T}_1=\mathcal{T}_1 \mathcal{T}_2= \mathcal{T}_1$.

Sufficiency. By Lemma \ref{lem:suboperator} (i), the equation \eqref{equ:difference1} holds. On the flip side, we have
\begin{equation*}
    \langle (\mathcal{T}_2 - \mathcal{T}_1) f, g \rangle_N = \langle \mathcal{T}_2  f, g \rangle_N - \langle  \mathcal{T}_1 f, g \rangle_N =  \langle   f, \mathcal{T}_2 g \rangle_N - \langle   f, \mathcal{T}_1g \rangle_N = \langle  f, (\mathcal{T}_2 - \mathcal{T}_1)g \rangle_N.
\end{equation*}
Thus, we have completed the proof.
\end{proof}

\begin{example}
For two hyperinterpolation operators $\mathcal{L}_n$ and $\mathcal{L}_m$ with $n \leq m$, $\mathcal{L}_n$ is a suboperator
of $\mathcal{L}_m$. Thus, the difference $\mathcal{L}_m - \mathcal{L}_n$ is also a hyper projection operator.  Specifically, we have
\begin{equation*}
  ( \mathcal{L}_m - \mathcal{L}_n)f= \sum_{\ell=n+1}^{m}\sum_{k=1}^{Z(d,\ell)} \left \langle {f, Y_{\ell k}^{(d)}} \right \rangle_N   Y_{\ell k}^{(d)}.
\end{equation*}
\end{example}

\section{Hyper ideals and hyper homomorphisms}\label{sec:ideals_homomorphism}
In this section, we mainly investigate the hyper ideals and hyper homomorphisms of hyperinterpolation operators. From the above, we know that 
both hard thresholding hyperinterpolation operator $\mathcal{H}_{n}^{\lambda}$ and hyperinterpolation operator $\mathcal{L}_n$ are
elements of the hyper semigroup $\mathfrak{B}(\mathbb{S}^d)$. Now we define three sets that can be useful in the latter, 
\begin{equation*}
    \mathfrak{I}(\mathbb{S}^d):=\{\mathcal{T}:  \mathcal{T} = \mathcal{L}_n \text{ or }\mathcal{T} = \mathcal{H}_n^{\lambda},  \mathcal{T} \in \mathfrak{B}(\mathbb{S}^d)  \}
\end{equation*}
and
\begin{equation*}
    \mathfrak{H}(\mathbb{S}^d):=\{\mathcal{T}:  \mathcal{T} = \mathcal{H}_n^{\lambda},  \mathcal{T} \in \mathfrak{B}(\mathbb{S}^d)  \},
\end{equation*}
and
\begin{equation*}
    \mathfrak{L}(\mathbb{S}^d):=\{\mathcal{T}:  \mathcal{T} = \mathcal{L}_n,  \mathcal{T} \in \mathfrak{B}(\mathbb{S}^d)  \}.
\end{equation*}
It is clear that $\mathfrak{I}(\mathbb{S}^d)$ and $\mathfrak{H}(\mathbb{S}^d)$ are commutative hyper semigroups, respectively. 

We mimic the definition of an \emph{ideal} \cite[p.16]{Grillet1995Semigroups}. 
\begin{definition}\label{def:ideal}
    A subset $\mathcal{S}$ of a commutative hyper semigroup $\mathcal{A}$ is said to be an \emph{hyper ideal} if $xy \in \mathcal{S}$ whenever $x \in \mathcal{A}$ and $y \in \mathcal{S}$.
\end{definition}
If $ \mathcal{S}\neq \mathcal{A}$, $\mathcal{S}$ is a \emph{proper} hyper ideal. A  proper hyper ideal $\mathcal{S}$ of a commutative hyper ideal $\mathcal{A}$ is \emph{prime} if $x,y \in \mathcal{A}$ such that $xy \in \mathcal{S}$, then either $x \in \mathcal{S}$ or $y \in \mathcal{S}$. \emph{Minimal} hyper ideals are proper hyper ideals that does not  contain any smaller proper ideals. 
Then we have the following important result.
\begin{theorem}\label{thm:ideal}
 The subset $\mathfrak{H}(\mathbb{S}^d)$ is a minimal prime hyper ideal of $\mathfrak{I}(\mathbb{S}^d)$.
\end{theorem}

\begin{proof}
We first prove that $\mathfrak{H}(\mathbb{S}^d)$ is a proper hyper ideal. Let $\mathcal{H}_n^{\lambda} \in \mathfrak{H}(\mathbb{S}^d)$ and $\mathcal{H}_m^{\lambda},\mathcal{L}_m \in \mathfrak{I}(\mathbb{S}^d)$, where $m$ may be not equal to $n$. Then we have
\begin{equation}\label{equ:prime_ideal_1}
    \mathcal{L}_m \mathcal{H}_n^{\lambda} = \left\{\begin{array}{ll}
        \mathcal{H}_n^{\lambda}, & m\geq n, \\
        \mathcal{H}_m^{\lambda}, & m < n ,
    \end{array}  \right.
\end{equation}
and
\begin{equation}\label{equ:prime_ideal_2}
    \mathcal{H}_m^{\lambda} \mathcal{H}_n^{\lambda} = \left\{\begin{array}{ll}
        \mathcal{H}_n^{\lambda}, & m\geq n, \\
        \mathcal{H}_m^{\lambda}, & m < n .
    \end{array}  \right.
\end{equation}
Indeed, for $\mathcal{L}_m \mathcal{H}_n^{\lambda}$ and $f \in \mathcal{C}(\mathbb{S}^d)$, we obtain
\begin{equation*}
 \begin{aligned}
    \mathcal{L}_m \mathcal{H}_n^{\lambda} f & = \sum_{\ell=0}^{m}\sum_{k=1}^{Z(d,\ell)} \left \langle  \sum_{\ell'=0}^{n}\sum_{k'=1}^{Z(d,\ell')} \eta_{H} \left( \left \langle f, Y_{\ell' k'}^{(d)} \right \rangle_N, \lambda  \right)Y_{\ell' k'}^{(d)}   , Y_{\ell k}^{(d)}\right \rangle_N Y_{\ell k}^{(d)} \\
     &=\sum_{\ell=0}^{m}\sum_{k=1}^{Z(d,\ell)} \sum_{\ell'=0}^{n}\sum_{k'=1}^{Z(d,\ell')} \eta_{H} \left( \left \langle f, Y_{\ell' k'}^{(d)} \right \rangle_N, \lambda  \right) \left \langle Y_{\ell'k'}^{(d)} , Y_{\ell k}^{(d)} \right \rangle _N Y_{\ell k}^{(d)}.
 \end{aligned}
\end{equation*}
Similarly, one can prove that
\begin{equation*}
    \mathcal{H}_m^{\lambda} \mathcal{H}_n^{\lambda} f = \sum_{\ell=0}^{m}\sum_{k=1}^{Z(d,\ell)} \eta_H \left( \sum_{\ell'=0}^{n}\sum_{k'=1}^{Z(d,\ell')} \eta_H \left( \left \langle f, Y_{\ell'k'}^{(d)}\right \rangle_N, \lambda \right) \left \langle Y_{\ell k}^{(d)}, Y_{\ell' k'}^{(d)} \right \rangle_N, \lambda \right) Y_{\ell k}^{(d)}.
\end{equation*}

By equations \eqref{equ:prime_ideal_1} and \eqref{equ:prime_ideal_2}, it is clear that $\mathfrak{H}(\mathbb{S}^d)$ is prime. Now, we prove that $\mathfrak{H}(\mathbb{S}^d)$ is minimal. Suppose that $\mathcal{S}=\{\mathcal{H}_1^{\lambda},\ldots,\mathcal{H}_k^{\lambda}\}$ with $1\leq k <n$.  It is obvious that $\mathcal{S}$ is a hype ideal of $\mathfrak{J}(\mathbb{S}^d)$, but is not prime. Although
$\mathcal{H}_{k+1}^{\lambda}, \mathcal{L}_k  \in \mathfrak{J}(\mathbb{S}^d)$ and $\mathcal{H}_{k+1}^{\lambda} \mathcal{L}_k = \mathcal{H}_{k}^{\lambda} \in \mathcal{S}$, both $\mathcal{H}_{k+1}^{\lambda} $ and $\mathcal{L}_k $ are not elements of $\mathcal{S}$.

Thus, we have completed the proof.
\end{proof}

\begin{remark}
    The hyper semigroup $\mathfrak{J}(\mathbb{S}^d)$ does not have a maximal hyper semigroup that is not contained in any larger proper ideal. For example, $\mathfrak{H}(\mathbb{S}^d)$ is a proper hyper semigroup of $\mathfrak{J}(\mathbb{S}^d)$ but is not maximal. Since we can easily prove that the subset $\{\mathcal{H}_1^{\lambda}, \ldots, \mathcal{H}_n^{\lambda},\mathcal{L}_1,\ldots,\mathcal{L}_k\}$ with $1\leq k <n$ of $\mathfrak{J}(\mathbb{S}^d)$ is also a proper hyper semigroup.
\end{remark}

We now define the \emph{hyper homomorphism}. For background on homomorphisms in semigroup theory, see \cite[p.9]{Grillet1995Semigroups}.
\begin{definition}\label{def:hyper_homomorphism}
Suppose $\mathcal{A}$ is a hyper semigroup and $\phi$ is a linear operator on $\mathcal{A}$ which is not identically zero. If
\begin{equation*}
    \phi(xy)=\phi(x)\phi(y)
\end{equation*}
for all $x, y\in \mathcal{A}$, then $\phi$ is called a \emph{hyper homomorphism} on $\mathcal{A}$.
\end{definition}

\begin{theorem}
Hyperinterpolation operator $\mathcal{L}_n$ is a hyper homomorphism on the hyper semigroup $\mathfrak{I}(\mathbb{S}^d)$.
\end{theorem}

\begin{proof}
Let $\mathcal{H}_{m_1}^{\lambda}, \mathcal{H}_{m_2}^{\lambda}, \mathcal{L}_{s_1},\mathcal{L}_{s_2} \in \mathfrak{I}(\mathbb{S}^d)$. Without loss of generality, we assume $m_i,s_i \leq n$ for $i=1,2$.
Elementary calculations yield the following three cases
\begin{enumerate}
    \item[(i)] for $\mathcal{L}_{s_1},\mathcal{L}_{s_2}$, we have
               \begin{equation*}
                \mathcal{L}_n(\mathcal{L}_{s_1}\mathcal{L}_{s_2})= \left\{\begin{array}{ll}
                    \mathcal{L}_n(\mathcal{L}_{s_1}) = \mathcal{L}_{s_1} = \mathcal{L}_n(\mathcal{L}_{s_1}) \mathcal{L}_n(\mathcal{L}_{s_2}), & s_1 < s_2, \\
                    \mathcal{L}_n(\mathcal{L}_{s_2}) = \mathcal{L}_{s_2} = \mathcal{L}_n(\mathcal{L}_{s_1}) \mathcal{L}_n(\mathcal{L}_{s_2}), & s_1 \geq s_2.
                \end{array}  \right.
               \end{equation*} 
    \item[(ii)] for $\mathcal{L}_{s_1},\mathcal{H}_{m_1}^{\lambda}$, we have
    \begin{equation*}
        \mathcal{L}_n(\mathcal{L}_{s_1}\mathcal{H}_{m_1}^{\lambda})= \left\{\begin{array}{ll}
            \mathcal{L}_n(\mathcal{H}_{m_1}^{\lambda}) = \mathcal{H}_{m_1}^{\lambda} = \mathcal{L}_n(\mathcal{L}_{s_1}) \mathcal{L}_n(\mathcal{H}_{m_1}^{\lambda}), & m_1 < s_1, \\
            \mathcal{L}_n(\mathcal{H}_{s_1}^{\lambda}) = \mathcal{H}_{s_1}^{\lambda} = \mathcal{L}_n(\mathcal{L}_{s_1}) \mathcal{L}_n(\mathcal{H}_{m_1}^{\lambda}), & m_1 \geq s_1.
        \end{array}  \right.
       \end{equation*} 
       The above relation for $\mathcal{H}_{m_1}^{\lambda}\mathcal{L}_{s_1}$ also holds since hard thresholding operator and hyperinterpolation operator are commutative.
    \item[(iii)] for $\mathcal{H}_{m_1}^{\lambda},\mathcal{H}_{m_2}^{\lambda}$, we have
    \begin{equation*}
        \mathcal{L}_n(\mathcal{H}_{m_1}^{\lambda}\mathcal{H}_{m_2}^{\lambda})= \left\{\begin{array}{ll}
            \mathcal{L}_n(\mathcal{H}_{m_1}^{\lambda}) = \mathcal{H}_{m_1}^{\lambda} = \mathcal{L}_n(\mathcal{H}_{m_1}^{\lambda}) \mathcal{L}_n(\mathcal{H}_{m_2}^{\lambda}), & m_1 < m_2, \\
            \mathcal{L}_n(\mathcal{H}_{m_2}^{\lambda}) = \mathcal{H}_{m_2}^{\lambda} = \mathcal{L}_n(\mathcal{H}_{m_1}^{\lambda}) \mathcal{L}_n(\mathcal{H}_{m_2}^{\lambda}), & m_1 \geq m_2.
        \end{array}  \right.
       \end{equation*}
\end{enumerate}
Thus, we have completed the proof.
\end{proof}

\section{Final remark}
In this paper, we explored the algebraic properties of the hyperinterpolation class on the unit sphere $\mathbb{S}^d \subset \mathbb{R}^{d+1}$, extending concepts from functional analysis into the hyperinterpolation framework under classical quadrature exactness assumptions \eqref{equ:semiinner}. Our main contributions can be summarized in two key aspects:

\begin{enumerate}
    \item[$\bullet$] \textbf{Establishment of Hyper Projection Operators:} We developed a theory of hyper projection operators within the context of a semi-inner product, laying a solid foundation for understanding the projection properties of the hyperinterpolation class.
    \item[$\bullet$] \textbf{Introduction of Hyper Semigroup:} We introduced the concept of hyper semigroup and demonstrated that the hyper semigroup containing hard thresholding hyperinterpolation operators forms the minimal prime hyper ideal within the broader hyper semigroup that includes both hard thresholding hyperinterpolation and classical hyperinterpolation operators. This elucidates the algebraic structures underlying the hyperinterpolation class.
\end{enumerate}

Looking ahead, leveraging the Marcinkiewicz-Zygmund inequality \cite{dai2010MZ,Filbir2011MZ,Mhaskar2001MZ,Zygmund1977trigonometric} offers a promising avenue for further exploration of the algebraic properties of the hyperinterpolation class, particularly in scenarios where quadrature exactness assumptions may not hold \cite{an2024bypassing}. By relaxing these assumptions, we can extend our understanding and applicability of hyperinterpolation.

Furthermore, future research could explore the underlying algebraic structures of hyper semigroups, as the intersection of semigroup theory \cite{Grillet1995Semigroups} and hyperinterpolation remains largely unexplored in the literature.

\section*{Acknowledgments}
The first author (C. An) of the research is partially supported by National Natural Science Foundation of China (No. 12371099) and Special Posts of Guizhou University (No. [2024]42).  The authors express their sincere thanks to Prof. Feng Dai for his helpful suggestions.

\bibliographystyle{siamplain}
\bibliography{refferences}
\end{document}